\documentclass[twoside]{irmaems}
\usepackage{amssymb} 
\usepackage{amsmath} 
\usepackage{latexsym}
\usepackage{enumerate}
\usepackage{labelfig}
\usepackage{amssymb}
\usepackage[cp850]{inputenc}
\usepackage{epsfig}
\usepackage{epstopdf}
\usepackage[mathscr]{eucal}

\renewcommand\theequation{\arabic{section}.\arabic{equation}}

\theoremstyle{definition} 

 \newtheorem{definition}{Definition}[section]
 \newtheorem{remark}[definition]{Remark}

\theoremstyle{plain}      

 \newtheorem{proposition}[definition]{Proposition}
 \newtheorem{theorem}[definition]{Theorem}
 \newtheorem{corollary}[definition]{Corollary}
 \newtheorem{lemma}[definition]{Lemma}
 \newtheorem{property}[definition]{Property}
 \newtheorem{question}[definition]{Question}

\newtheorem*{conjecture}{Conjecture}

\newcommand{\arcsinh}{\operatorname{arcsinh}}

\newcommand{\R}{{\mathbb R}}

\newcommand{\Hyp}{{\mathbb H}}

\newcommand{\Z}{{\mathbb Z}}

\newcommand{\BB}{{\mathcal B}}

\newcommand{\T}{{\mathcal T}}
\newcommand{\PP}{{\mathcal P}}

\newcommand{\Q}{{\mathcal Q}}

\newcommand{\PSL}{{\rm PSL}}
\newcommand{\equal}{{\mathcal E}}

\newcommand{\area}{\mathrm{area}}
\newcommand{\sys}{{\rm sys}\,}

\begin{document}

\title{Simple closed geodesics and the study of Teichm\"uller spaces}

\author{Hugo Parlier 
}

\address{Department of Mathematics, University of Toronto, Canada\\
}

\maketitle




\tableofcontents   

\section{Introduction}\label{sect:intro}

The goal of the chapter is to present certain aspects of the relationship between the study of simple closed geodesics and Teichm\"uller spaces. The set of simple closed geodesics is more than a mere curiosity and has been central in the study of surfaces for quite some time: it was already known to Fricke that a carefully chosen finite subset of such curves could be used as local coordinates for the space of surfaces. Since then, the literature on the subject has been vast and varied. Recent results include generalizations of McShane's Identity \cite{mc98,mi071,mi072} and results on how to use series based on lengths of simple closed geodesics to find invariant functions over Teichm\"uller space to calculate volumes of moduli spaces. Questions surrounding multiplicity in the simple length spectrum are sometimes related to questions in number theory \cite{ha86,se85}. In a somewhat different direction, and although this theme will not be treated here, a related subject is the combinatorics of simple closed curves. The geometry of the curve complex \cite{chrase08,masmin99,ra05,ra07,sehand} and the pants complex \cite{br03} have played an important role in the study of the large scale geometry of Teichm\"uller spaces with its different metrics and the study of hyperbolic $3$-manifolds.\\

It should be noted that this chapter should not in any way be considered a survey of all relationships between Teichm\"uller space and simple closed geodesics, but just a presentation of certain aspects the author is familiar with. Specifically, this chapter will concentrate on two themes.\\

The first theme is the study of the set of simple closed geodesics in contrast with the set of closed geodesics. There are a number of ways in which these sets differ and these illustrate the special nature of simple closed geodesics. Three subjects of contrast are exposed here. The first concerns results related to the non-density of simple closed geodesics and in particular the Birman-Series theorem \cite{bise85}. The second subject is about the contrast in growth of the number of simple closed geodesics in comparison with closed geodesics, and in particular Mirzakhani's theorem \cite{mi08}. The third subject concerns how multiplicity differs in the full length spectrum in comparison with the simple length spectrum.\\

The second theme is on systoles, their lengths, and other related quantities such as the lengths of pants decompositions. For systolic constants, many of the known results are due to Schmutz Schaller who wrote a survey article on the subject \cite{sc98}, and so the information provided is intended to somewhat complement his exposition. Bers' constants, introduced by Bers \cite{be74,be85} are upper bounds on lengths of a shortest pants decomposition of a surface, and have been extensively studied by Buser, who proves a number of bounds on these \cite{buhab,bubook}. Here again, the information provided should be seen as a complement of \cite{bubook}. There are a certain number of parallels and similarities between the two problems, and one of the goals is to illustrate this.

\section{Generalities}\label{sec:gen}

Consider Teichm\"uller space $\T_{g,n}$, the set of marked finite area hyperbolic metrics on an orientable surface of genus $g$ with $n$ marked points. Surfaces with a non-empty set of cusps will called punctured surfaces. On occasion, we will talk about non-complete surfaces with boundary geodesics or cusps instead of just cusps, and they will be called surfaces with holes.\\

A free homotopy class of a closed curve is called non-trivial if it is not homotopic to a disk or to a puncture. It is called simple if it can be represented by a simple closed curve. We begin with the following essential property.

\begin{property}\label{pro:unique} A non-trivial closed curve is freely homotopic to a unique closed geodesic. If the closed curve is simple then so is the freely homotopic closed geodesic.
\end{property}

One way of seeing this is by considering the lifts of a non-trivial curve to the universal cover $\Hyp$. The lifts are all disjoint simple curves between boundary points of $\Hyp$. The geodesics between the same boundary points are invariant by the Fuchsian group, and the quotient is the desired simple closed geodesic. The same argument shows that the closed geodesic in the free homotopy class of a closed curve always minimizes its self-intersection number. Simple closed geodesics will be generally thought of as primitive and unoriented.\\

The above property allows us to associate to a homotopy class $[\alpha]$ a length function, which to a surface $M\in \T$ associates the length $\ell_M(\alpha)$ of the unique geodesic in the homotopy class of $\alpha$. A fundamental property of these length functions is that they are analytic, with respect to the usual analytic structure on Teichm\"uller space \cite{abbook}. For a given surface $M$, we shall denote $\Delta(M)$ the marked set of lengths of closed geodesics. By the above property, this is a countable set of values. The simple length spectrum $\Delta_0(M)$ is the subset of $\Delta(M)$ restricted to simple closed geodesics.\\

Simple closed geodesics provide useful parameters for Teichm\"uller space via pants decompositions. A pants decomposition is a collection of disjoint simple closed geodesics of the surface such that the complementary region is a collection of three holed spheres (pants). By a simple topological argument one sees that there are exactly $3g-3+n$ curves in a pants decomposition, and this number is sometimes referred to as the complexity of the surface. This provides us with $3g-3+n$ length functions, which although they do not determine the surface, they determine the geometry on the complement of the pants decomposition geodesics. To determine a marked surface, one uses a twist parameter to determine how the pants curves are glued together. Generally one measures twist by looking at perpendicular geodesic arcs on pants between distinct boundary curves. By cutting along these, one obtains a pair of symmetric hyperbolic right-angled hexagons. Twist is then measured by the displacement factor between two perpendicular arcs intersecting the boundary. The collection of lengths and twists are called the Fenchel-Nielsen coordinates.\\

A second consequence of the decomposition of a hyperbolic surface into hexagons is the well known collar lemma (see \cite{bu78,bubook,kee74,ra79} for different versions).

\begin{lemma}[Keen's Collar Lemma]\label{lem:collar}
Around a simple closed geodesic $\gamma$ there is always an embedded hyperbolic cylinder (called a collar) of width 
$$
w(\gamma)=\arcsinh\left(\frac{1}{\sinh\frac{\ell(\gamma)}{2}}\right).
$$
Furthermore, the collars around pants decomposition geodesics are all disjoint. 
\end{lemma}

The proof essentially follows from the above discussion and hyperbolic trigonometry (see \cite{beabook,bubook} for hyperbolic trigonometry formulas, and \cite{fenibook} for the original Fenchel and Nielsen approach). Indeed, on each side of a simple closed geodesic $\gamma$, one obtains two isometric right angled hexagons with one of their side lengths $\ell(\gamma)/2$. By hyperbolic trigonometry, the subsequent edges cannot have length less than $\arcsinh\left(\frac{1}{\sinh\frac{\ell(\gamma)}{2}}\right)$ and the result follows.\\

Although Fenchel-Nielsen coordinates are very useful, they do not provide a homogeneous set of parameters. However, if one allows more length functions of simple closed curves, one does obtain a complete local description of Teichm\"uller space.

\begin{theorem}\label{thm:proj} There is a fixed finite set of simple
closed geodesics $\gamma_1,\ldots,\gamma_m$ such that the map
$\varphi:$
$M\mapsto(\ell_M(\gamma_1),\ldots,\ell_M(\gamma_m))$ is
projectively injective on $\T_{g,n}$.
\end{theorem}

Recall that a map $f: X \rightarrow V$ where $V$ is a real vector space is projectively injective if $f(x) = t f(y),$ for some $t \in \R$, implies $x=y$. Interestingly, this theorem fails to be correct in all generality if one allows surfaces with variable boundary length \cite{pamcmul}. There are different versions of it and as stated has probably been known for some time. We refer to \cite{ham031, ham032, sc932, sc013} for different statements about the $m$ and related questions if one allows the curves to not be simple and whether one wants a projectively injective map or just an injective map.\\

A tool that will come up several times in this discussion is the following, which will be called the length expansion lemma.

\begin{lemma}[Length expansion lemma] Let $S$ be a surface with $n > 0$ boundary curves $\gamma_1,\hdots,\gamma_n$. For
$(\varepsilon_1,\hdots,\varepsilon_n) \in (\R^+)^n$ with at least
one $\varepsilon_i\neq 0$, there exists a surface $\tilde{S}$ with
boundary geodesics of length
$\ell(\gamma_1)+\varepsilon_1,\hdots,\ell(\gamma_n)+\varepsilon_n$ such that
all corresponding simple closed geodesics in $\tilde{S}$ are of
length strictly greater than those of $S$.
\end{lemma}

There are different proofs and versions of this lemma \cite{pa051,thspine}. Recently, Papadopoulos and Th\'eret \cite{path101} proved a stronger version where they show that not only can one increase the lengths of all simple closed geodesics but one can do so such that the infimum of the ratios of lengths between the long geodesics and the short geodesics is strictly greater than $1$. They use this to show things about how one can and cannot generalize Thurston's asymmetric metric defined in \cite{th98} to surfaces with boundary (see also \cite{path07} for an overview on this metric). 

\section{Simple closed geodesics versus the set of closed geodesics}

One often studies the behavior of the set of simple closed geodesics in contrast with the set of closed geodesics. This section is devoted to showing that simple closed geodesics are rare in the set of closed geodesics in several ways.

\subsection{The non-density of simple closed geodesics}

One of the first remarkable results in this direction is a theorem of Birman and Series. It is well known that on hyperbolic surfaces, points lying on the set of closed geodesics form a dense subset of the surface. In fact, they are even dense in the tangent bundle. Birman and Series \cite{bise85} show that this is in sharp contrast with the set of simple closed geodesics (and more generally with the set of simple complete geodesics).

\begin{theorem}[Birman-Series]\label{thm:BS}
The set of points lying on a simple complete geodesic is nowhere dense and has Haussdorf measure one.
\end{theorem}

More generally, they show this to be true for complete geodesics with bounded self-intersection number. One of the essential steps in their approach is he following. For a given $L$, the set of simple complete geodesics lie in an $\epsilon$ neighborhood of some set of geodesic arcs of length less than $L$. They show that for any given surface, there are positive constants $L,C,\alpha$ and a polynomial $P$ such that the full set of simple complete geodesics lies in the $\varepsilon= Ce^{-\alpha n}$ neighborhood of a set of at most $P(n)$ geodesic arcs of length at most $L$. This has to do with the fact that ``long" simple complete geodesics spend great deals of time running parallel to themselves. Along similar lines, there have been different descriptions of algorithms which determine whether a word in the fundamental group corresponds to a simple closed geodesic, including one by Birman and Series \cite{bise842}.\\

In particular one can ask how ``non-dense" the simple closed geodesics are on surfaces. In \cite{pabugaps} the following theorem is shown.

\begin{theorem}\label{thm:gaps}
There exists a constant $c_g > 0$, depending only on $g$, such that any compact Riemann surface $M$ of genus $g$ contains a disk of radius $c_g$ into which the simple closed geodesics do not enter. Reciprocally, for any $\varepsilon>0$ there exists a surface $S_{\varepsilon}$ on which the geodesics are $\varepsilon$-dense.
\end{theorem}

To show the existence of the constants, one uses theorem \ref{thm:BS} and one needs to show the continuity of the ``gaps" in Teichm\"uller space, use the compactness of the thick part of Teichm\"uller space and a discussion of the constant behavior in the thin part. The converse is essentially a consequence of a theorem of Scott's \cite{sco78,sco85} which says that given any closed geodesic on a surface, there exists a finite cover of the surface where all of the primitive lifts of the original closed geodesic are simple. Rivin has recently asked \cite{ri09} if one can quantify Scott's result, i.e., compute or at least find bounds on the minimal degree cover which is necessary to ``unravel" a closed geodesic with $k$ self-intersections.\\

\subsection{The growth of the number of simple closed geodesics}
A recent result of Mirzakhani's computes the asymptotic growth of the number of simple closed geodesics of less than a given length, and as we shall outline, this provides another ``simple closed geodesics are rare" analogy. If one counts the number of closed geodesics on hyperbolic surfaces of length less than $L$, the asymptotic result does not depend on the surface, or even on the topology of the surface. Quite remarkably, this number always behaves asymptotically like $\frac{e^L}{L}$ (this is sometimes referred to as Huber's asymptotic law \cite{hu59}). In contrast, Mirzakhani has shown the following.

\begin{theorem}[Mirzakhani]\label{thm:growth}
Let $\mathscr{N}_M(s,L)$ be the number of simple closed geodesics of lengths ${}\leq L$ on $M$. For $L \to \infty$ this number has the asymptotic behavior
 \begin{equation*}%
\mathscr{N}_M(s,L) \sim c_M L^{6g-6},
 \end{equation*}%
where $c_M$ is a constant depending on $M$.
 \end{theorem}%

Mirzakhani's theorems show more than the above stated result, and in particular it is shown that the leading coefficient $c_M$ gives a continuous proper function over moduli space. One of the ideas of the proof is to notice that up to action of the mapping class group, there are only a finite number of types of simple closed curves on a surface. For each type, one can count the growth of a simple closed geodesic under the action of the mapping class group (which she does) and then obtain the result by adding the finite number of types. There is quite a rich history to this problem which traces back to Dehn. Rivin \cite{ri01} had previously obtained partial and related results. He also provided a simplified proof \cite{ri05} and explains some of the history of the problem.\\

Also prior to Mirzakhani's work, was a theorem of McShane and Rivin \cite{mcri951,mcri952} in the particular case of a once punctured torus. They obtained the same theorem, and showed \cite{mcri951} that the leading asymptotic coefficient correspond has to do with the stable norm on the homology of the torus. Indeed, in the case of once punctured tori, homology classes and oriented not necessarily primitive homotopy classes of simple closed curves coincide. Given a metric on a torus $T$, the length of the corresponding geodesics induces a norm on the integer homology of a torus $H_1(T,\Z)$ which can be extended to a norm on $H_1(T,\R)$. What they show is that the leading asymptotic coefficient is in fact the inverse of the area of the unit ball of this norm. They conjecture that the leading coefficient is maximal (i.e. the area of the unit ball is minimal) for the modular torus. The modular torus is the unique once punctured torus with an isometry group of order $12$, and it can be obtained for instance by a quotient of $\Hyp$ by an index $3$ subgroup of $\PSL_2(\Z)$. Similarly, one could ask the same questions about the constants that appear in as the leading asymptotic coefficients of Mirzakhani's formula. 

\subsection{Multiplicities of simple closed geodesics}

A well known theorem of Randol \cite{ra80}, using a construction due to Horowitz, states that multiplicity is unbounded in the set of lengths of closed geodesics. The proof is by construction, and in fact one constructs arbitrarily large sets of homotopy classes of closed curves whose geodesic representatives all have the same length, regardless of the hyperbolic metric on the surface (the length of course does depend on the metric, but the lengths are always equal). There is a nice illustrated proof of Randol's result in \cite{bubook}.\\

Leininger \cite{le03} has shown that these curves, sometimes called equivalent curves, have the property that they intersect all simple closed geodesics the same number of times. Essentially, this follows from the collar lemma. However, somewhat surprisingly, the converse fails to be correct and this is more subtle. Note that this is again in contrast with the set of simple closed geodesics. Indeed Thurston \cite{th79} uses the fact that homotopy classes of simple closed curves are determined by their intersection numbers with other simple closed curves to construct his compactification of Teichm\"uller space.\\

In contrast one can study multiplicity in the simple length spectrum. The following is true \cite{pamcmul}.

\begin{theorem}\label{thm:baire}
The set of surfaces with all simple closed geodesics of distinct length is dense in Teichm\"uller space and its complement is Baire meagre.
\end{theorem}

As its statement suggests one can show this by using the Baire category theorem. Denote $\equal(\alpha,\beta)$ the set of all surfaces where a distinct pair homotopy classes of simple closed curves have geodesic representatives of the same length. Because length functions are analytic, these are zero sets of length functions, and are thus closed. Also, these sets have no interior, because otherwise they would be equal over all Teichm\"uller space. Using the collar theorem, one can show this to be impossible. The set of surfaces with all simple closed geodesics of distinct length is the complement of the union $\equal$ of the sets $\equal(\alpha,\beta)$. The result then follows from the Baire category theorem as there are a countable number of homotopy classes.\\

One can actually say more about the topology of the sets $\equal(\alpha,\beta)$ and their union $\equal$ \cite{pamcmul}.

\begin{theorem}
The sets $\equal(\alpha,\beta)$ are connected sub-manifolds of
Teichm\"uller space. The set $\equal$ is connected.
\end{theorem}

To show that these are indeed sub-manifolds and not just sub-varieties, one uses Thurston's stretch maps \cite{th98}. The connectedness of $\equal$ is a consequence of what follows.\\

A natural question to ask is whether one can deform a surface within $\equal^c$. In fact not \cite{pamcmul}.

\begin{theorem}\label{thm:noarc}
If ${\mathcal A}$ is a path in Teichm\"uller space then there is
a surface on ${\mathcal A}$ which has at least two distinct simple closed
geodesics of the same length.
\end{theorem}

As a consequence, the set $\equal$ is the complement of a totally disconnected set of Teichm\"uller space and as such is connected. The proof of this uses the projectively injective map to Teichm\"uller space described previously. Given two distinct surfaces $M_1$ and $M_2$, one constructs pairs of curves $\alpha,\beta$ such that $\ell_{M_1}(\alpha)>\ell_{M_1}(\beta)$ and $\ell_{M_2}(\alpha)<\ell_{M_2}(\beta)$ and the construction relies on the projectively injective map. Now by continuity of length functions, along any path between $M_1$ and $M_2$ there is a point on which the two simple closed geodesics have the same length. As a corollary of the above discussion one obtains the following.

\begin{corollary} The marked order in lengths of simple closed geodesics determines a unique point in $\T_{g,n}$.
\end{corollary}

One of the motivations of this study was to study the nature of the following conjecture, attributed to Rivin \cite{sc98}.

\begin{conjecture}[Rivin]
Multiplicity in the simple length spectrum is bounded above by a constant which depends on the topology of the surface.
\end{conjecture}

It should be mentioned that lower bounds which depend on topology have been established in the work of Schmutz Schaller in his investigation of surfaces with a large number of systoles (see section \ref{sec:systoles}). Also due to Schmutz Schaller is the following conjecture \cite{sc98}.

\begin{conjecture}[Schmutz Schaller]\label{con:ss}
Multiplicity is bounded by $6$ in the particular case of once-punctured tori.
\end{conjecture}

In fact, this is a geometric generalization to the well-known Markov uniqueness conjecture in number theory \cite{fr13} which is in fact equivalent to whether Schmutz Schaller's conjecture is satisfied by the modular torus. 

\begin{conjecture}[Frobenius] A solution $(a,b,c)$ of positive integers to the Markov cubic
\begin{eqnarray}
a^{2} + b^{2}+ c^{2} - 3abc = 0
\end{eqnarray}
is uniquely determined by $\max\{a,b,c\}$.
\end{conjecture}

We refer the reader to \cite{pamcmuls,sc98} and references therein on known results concerning this conjecture. Let us note however that the apparent difficulty of the Markov uniqueness conjecture suggests that Schmutz Schaller's conjecture is very difficult. Furthermore, Schmutz Schaller \cite{sc98} remarked that there were no known surfaces where one knows that multiplicity in the simple length spectrum was even bounded. Note that by theorem \ref{thm:baire} most surfaces have simple multiplicity bounded by $1$, but whether or not a given surface has this property can be a difficult question. In \cite{pamcmuls}, a family of examples of one holed tori with multiplicity bounded by $6$ are given and will be discussed in propostion \ref{prop:maxtorus} and remark \ref{rem:modtori}. However, an example of either a punctured or a closed surface with bounded simple multiplicity remains to be shown.

\section{Short curves: systolic and Bers' constants}

The simple length spectrum has not nearly been studied as much as the full length spectrum. However, one particular length, namely the length of a shortest closed geodesic or the systole length, has been the object of numerous articles and investigations. The study of this function, largely developed by Schmutz Schaller (see for instance \cite{sc98} and references therein), continues to be a subject of active study. Another length function which has proved useful in the study of Teichm\"uller space is the length of the shortest pants decomposition of a surface. Both of these quantities are bounded by constants that depend only on the topology of the surfaces (and not on the metrics themselves). Here we'll try to compare some of the techniques and results used to study both problems.

\subsection{Systolic constants}\label{sec:systoles}

Consider a hyperbolic surface of genus $g$ with $n$ cusps, or if one is ambitious, with $n$ boundary geodesics. We define the systole to be the length of the shortest closed geodesic which does not belong to the boundary. Unless the surface is a pair of pants, such a curve is always simple. We define the systole function $\sys(S)$ to be the length of the systole of a surface. This gives an interesting function over Teichm\"uller space as is portrayed by the following theorem of Akrout \cite{ak03}.

\begin{theorem}[Akrout] 
The systole of Riemann surfaces is a topological Morse function on the Teichm\"uller space.
\end{theorem}

Note that this is the best one could hope for as it could not be a regular Morse function because the curve realizing systole length changes as one moves around Teichm\"uller space. To show this, Akrout shows a more general result concerning generalized systole functions which are functions on manifolds defined locally as the minimum of a finite number of smooth functions. He shows that if the manifold admits a connexion such that the Hessians of length functions are positive definite, such a function is a topological Morse function such that a Morse point of index $r$ is a eutatic point of rank $r$. Let us outline why this proves the above theorem. The manifold in our case is Teichm\"uller space endowed with the Weil-Petersson metric. In a neighborhood of a point, length functions are continuous and length spectra are discrete, there are only a finite number of simple closed curves that will realize the systole in the neighborhood, so the systole function satisfies the criteria of generalized systole functions. Furthermore, it is a result of Wolpert \cite{wo87} that length functions have positive definite Hessians with respect to the Weil-Petersson metric, and thus the result follows.\\

Although we begin the study of the systolic geometry of surfaces by Akrout's theorem, it should be noted that this is not chronologically correct. Schmutz Schaller had previously obtained partial results in this direction \cite{sc98,sc99} and had largely initiated a systematic study of systoles on hyperbolic surfaces and established a parallel with $n$-dimensional sphere packings. Bavard \cite{bav97}, by generalizing the study of Hermite invariants to systolic problems on manifolds, provided a theoretical framework to study extremal points of systole type functions. Using these parallels, Akrout's result is related to Ash's results \cite{as77,as80}. As in the case of sphere packings or Hermite invariants, one is interested in extremal points, local and global maxima, and these turn out to be difficult to find. First of all, by Mumford's compactness theorem, for each genus $g\geq 2$, there is a global maximum for the systole, which we shall denote $\sys(g,n)$ and $\sys(g)$ for closed surfaces.\\

\noindent{\it Exact values of systolic constants}\\

In the non-compact case, Schmutz Schaller proved that surfaces corresponding to quotients of $\Hyp$ by principal congruence subgroups of $\PSL_2(\Z)$ are in fact global maxima for the systole length of their corresponding signature. There are a number of other results known about low complexity cases, even with (fixed length) boundary geodesics \cite{sc931}.\\

For closed surfaces there is only one known value, in genus 2, a result of Jenni \cite{je84}. 

\begin{theorem}[Jenni] A systole $\sigma$ of a genus $2$ surface satisfies
$$
\cosh\frac{\ell(\sigma)}{2}\leq \sqrt{2}+1
$$
and equality occurs for a unique surface (up to isometry).
\end{theorem}

The surface that attains the optimal bound in genus 2 is the so-called Bolza curve, and is also the genus 2 surface with the highest number of conformal self-isometries (it has $48$). Since then there have been other proofs of Jenni's result, namely \cite{sc931} where among many results, one finds a complete list of all critical values of $\sys$ in genus 2.\\

\begin{figure}[h]
\begin{center}
\includegraphics[width=12cm]{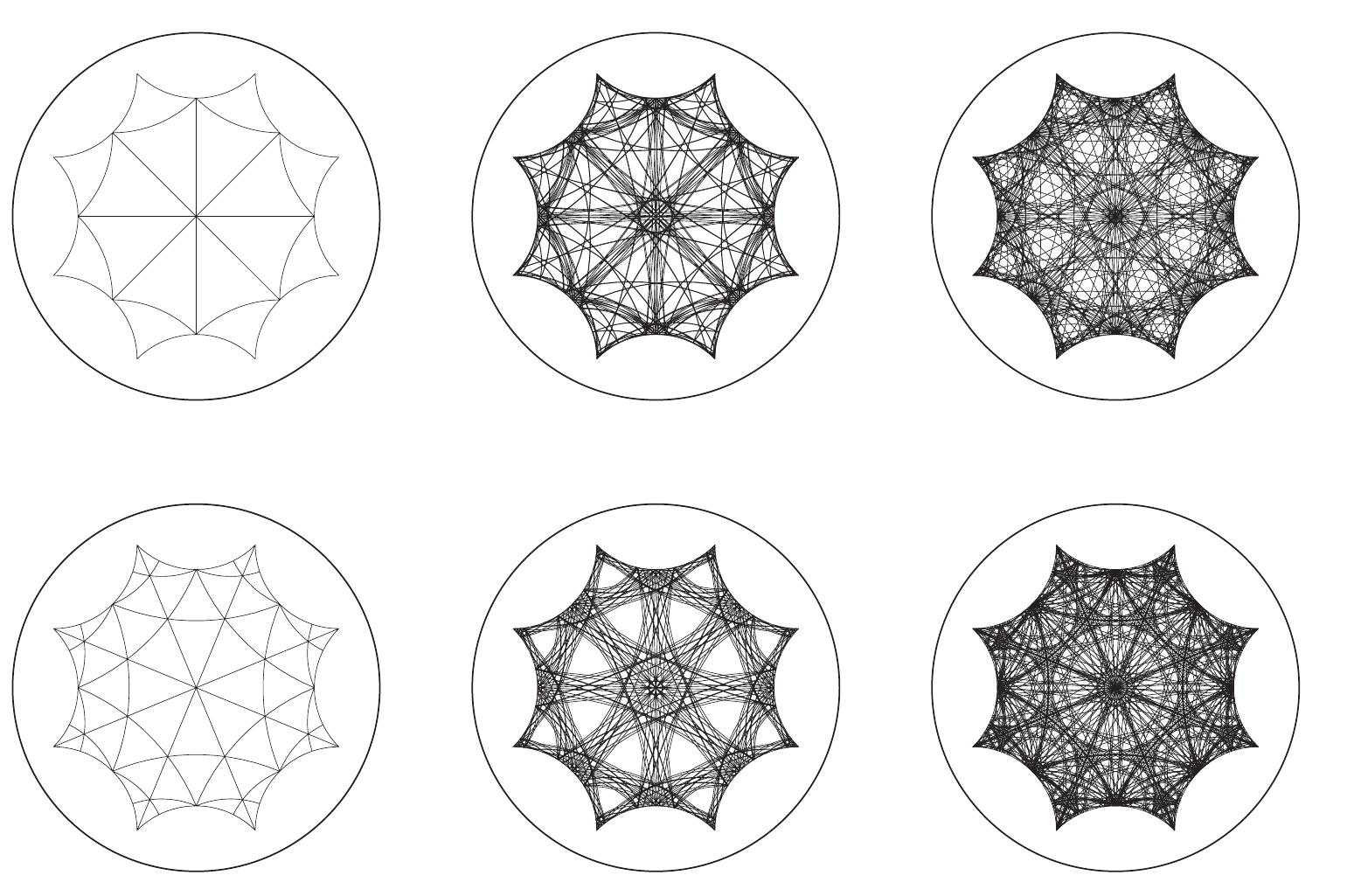}
\caption{The Bolza surface: on the upper left, one sees the 12 systoles, and on the lower right one sees the gaps from theorem \ref{thm:gaps}. See Peter Buser's article \cite{buanogia} for a description of how this picture was made.}
\label{fig:CV}
\end{center}
\end{figure}

An another proof is given by Bavard \cite{bav92} where the problem of the maximal systole among hyperelliptic surfaces is considered. He finds the maximal values in genus 2 and 5, and describes the surfaces which attain the bounds (and shows that they are unique up to isometry). As all genus 2 surfaces are hyperelliptic, the result coincides with Jenni's. For higher genus hyperelliptic surfaces, consider the quotient of the surface by the hyperelliptic involution. One obtains a hyperbolic sphere with $2g+2$ cone points of angle $\pi$ (corresponding to the Weierstrass points of the surface). A simple geodesic path of length $\ell$ between two distinct cone points lifts to a simple closed geodesic passing through two Weierstrass points on the closed surface above, of length $2\ell$. Finding optimal bounds for this problem is a type of hyperbolic equivalent to the well known Tammes problem on the Euclidean sphere (only known for certain numbers of points), and there is no reason to believe that it is any easier.\\

In genus 3, Klein's quartic is the surface with the most conformal self-isometries (it attains Hurwitz's upper bound of $84(g-1)=168$ but it fails to be maximum for the systole function. Schmutz Schaller has conjectured that another surface, also with a large number of symmetries, is maximal (the so-called $M(3)$ surface explicitly described in \cite{sc931}). 

\begin{conjecture}[Schmutz Schaller] A systole $\sigma$ of a surface in genus $3$ satisfies $\cosh\frac{\ell(\sigma)}{2}\leq 2+\sqrt{3}$ with equality occurring for a unique surface up to isometry.
\end{conjecture}

Although this is a yes or no type question, he has attained partial results which include the fact that it is a local maximum \cite{sc931} and that certain subsurfaces of this surface (with boundary) are optimal in their configuration. In order to give an idea on one of this last result, consider the following toy problem. one considers a one-holed torus with a boundary geodesic of length $\ell$. Among all such surfaces, which one has maximal systole length? The solution to this problem requires only cut and paste techniques and hyperbolic trigonometry. This same proposition, expressed differently and without proof, can be found in \cite{sc931}. We give a full proof as an illustration of some of the difficulties of these techniques in higher surface complexity.

\begin{proposition}\label{prop:maxtorus}
Let $\Q$ be a surface of signature $(1,1)$ with boundary geodesic
$\gamma$. Then $\Q$ contains a simple closed geodesic $\delta$
satisfying
$$
\cosh  \frac{\ell(\delta)}{2} \leq \cosh \frac{\ell(\gamma)}{6}+
\frac{1}{2}.
$$
This bound is sharp and for a given length of $\gamma$, there is a unique surface up to isometry which reaches this bound. 
\end{proposition}

\begin{proof}
The idea of the proof is to use hyperbolic polygons to obtain an
equation from which we can deduce the sharp bound.

Let $\delta$ be the shortest closed geodesic on $\Q$
with boundary geodesic $\gamma$. Let $c$ be the length of the perpendicular geodesic arc from
$\delta$ to $\delta$ on the embedded pants $(\delta,\delta,\gamma)$
obtained by cutting $\Q$ along $\delta$. By the formula for a right angled hexagon, $c$ is given by the following formula:
$$
\cosh c = \frac{\cosh \frac{\ell(\gamma)}{2} +
\cosh^2\frac{\ell(\delta)}{2}}{\sinh^2\frac{\ell(\delta)}{2}}.
$$
Another way of expressing it is using one of the four isometric
hyperbolic pentagons that form a symmetric pair of pants as in the
following figure.

\begin{figure}[h]
\leavevmode \SetLabels
\L(.28*.9) $\delta$\\
\L(.71*.9) $\delta$\\
\L(.48*.28) $\gamma$\\
\L(.43*.94) $\frac{c}{2}$\\
\L(.54*.94) $\frac{c}{2}$\\
\endSetLabels
\begin{center}
\AffixLabels{\centerline{\epsfig{file = 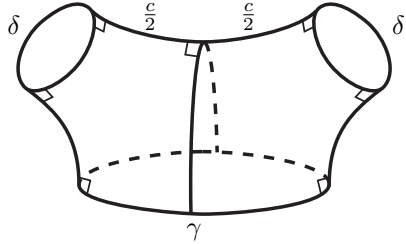, width=5cm,
angle= 0}}}
\end{center}
\vspace{-1.8cm}
\caption{A symmetric pair of pants} \label{fig:Vq1}
\end{figure}

The pentagon formula implies that
$$
\cosh^2 \frac{c}{2} = \frac{\cosh^2\frac{\ell(\gamma)}{4} +
\cosh^2\frac{\ell(\delta)}{2} - 1}{\cosh^2\frac{\ell(\delta)}{2} - 1}.
$$
The smaller $c$ is, the longer $\delta$ is. Thus to find an
upper bound on $\delta$, we need to find a minimal $c$. Consider
$\delta'$ the second shortest simple closed geodesic. Its not too difficult to see that $\delta'$ intersects
$\delta$ exactly once. For a given $\delta$ and $\gamma$, this $\delta'$
is of maximal length when $\Q$ is obtained by pasting
$\delta$ with a half twist.

The length of this maximal $\delta'$
can be calculated in the following quadrilateral.

\vspace{-1cm}
\begin{figure}[h]
\leavevmode \SetLabels
\L(.24*.20) $\frac{\ell(\delta)}{4}$\\
\L(.72*.68) $\frac{\ell(\delta)}{4}$\\
\L(.38*.13) $\frac{\ell(\delta')}{2}$\\
\L(.54*.74) $\frac{\ell(\delta')}{2}$\\
\L(.38*.55) $\frac{c}{2}$\\
\L(.59*.40) $\frac{c}{2}$\\
\endSetLabels
\begin{center}
\AffixLabels{\centerline{\epsfig{file = 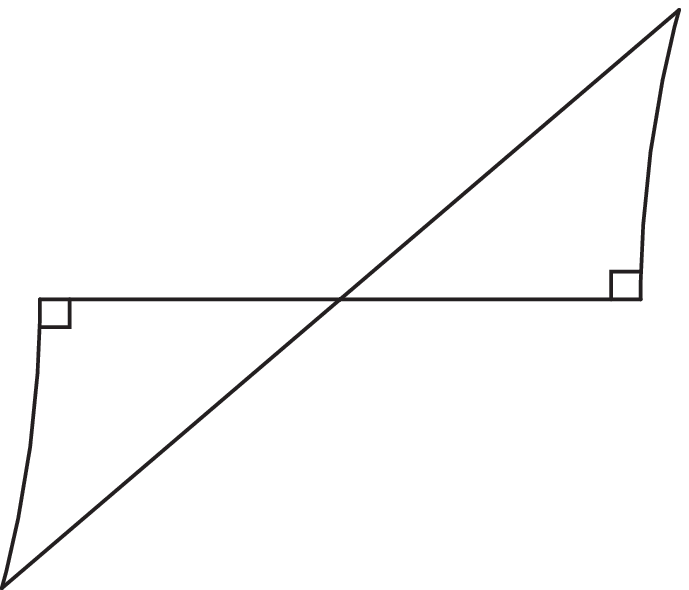, width=5cm,
angle= 0}}}
\end{center}
\vspace{-1cm}
\end{figure}

\vspace{-0.8cm}
From one of the two right-angled triangles that compose the
quadrilateral we have
$$
\cosh \frac{\ell(\delta')}{2} = \cosh \frac{c}{2}\cosh \frac{\ell(\delta)}{4}.
$$
Using the fact the $\delta\leq \delta'$ we can deduce
\begin{eqnarray*}
\cosh^2 \frac{\ell(\delta)}{2} & \leq & \cosh ^2 \frac{c}{2} \cosh^2
\frac{\ell(\delta)}{4}\\
& = & \frac{\cosh^2\frac{\ell(\gamma)}{4} + \cosh ^2
\frac{\ell(\delta)}{2}-1}{\cosh^2 \frac{\ell(\delta)}{2}-1}
\cosh^2\frac{\ell(\delta)}{4}\\
&=&
\frac{\cosh^2\frac{\ell(\gamma)}{4}+\cosh^2\frac{\ell(\delta)}{2}-1}{2(\cosh\frac{\ell(\delta)}{2}-1)}.
\end{eqnarray*}
From this we obtain the following condition:
$$
2\cosh^3\frac{\ell(\delta)}{2}-3\cosh^2\frac{\ell(\delta)}{2}+1 -
\cosh^2\frac{\ell(\gamma)}{4} \leq 0.
$$
With $x= \cosh \frac{\ell(\delta)}{2}$ and $C=
\cosh^2\frac{\ell(\gamma)}{4}>1$ we can study the following degree $3$
polynomial
$$
f(x)=2x^3-3x^2+1-C
$$
and find out when it is negative for $x>1$. The function $f$
satisfies $f(1)=-C<0$ and $f'(x)>0$ for $x>1$. The sharp condition
we are looking for is given by the unique solution $x_3$ to
$f(x)=0$ with $x>0$. Thus
$$
x'_3 = \frac{1}{2}(-1 + 2C + 2\sqrt{-C+C^2})^\frac{1}{3} +
\frac{1}{2} \frac{1}{(-1 + 2C + 2\sqrt{-C+C^2})^{\frac{1}{3}}} +
\frac{1}{2}.
$$
Now we replace $x$ and $C$ by their original values. Using
hyperbolic trigonometry we can show
$$
\cosh \frac{\ell(\delta)}{2} \leq \frac{1}{2}((\cosh
\frac{\ell(\gamma)}{2}+\sinh \frac{\ell(\gamma)}{2})^\frac{1}{3} +
(\cosh \frac{\ell(\gamma)}{2}+\sinh
\frac{\ell(\gamma)}{2})^{-\frac{1}{3}} + 1)
$$
which in turn can be simplified to
$$
\cosh  \frac{\ell(\delta)}{2} \leq \cosh \frac{\ell(\gamma)}{6}+
\frac{1}{2}.
$$
The bound is sharp and the length of $\delta$ satisfies this bound if and only if $\ell(\delta')=\ell(\delta)$ and the twist parameter $\delta$ is pasted with a half twist. \end{proof}

\begin{remark}\label{rem:modtori}For fixed boundary length, the unique surface that reaches the upper bound on systole length has exactly $3$ distinct systoles: the curves $\delta$, $\delta'$ and a third $\delta''$ which is the for instance the mirror image of $\delta'$ reflected along $\delta$. If the boundary is a cusp, then the surface obtained is in fact the modular torus mentioned previously, conformally equivalent to the torus obtained by taking a regular euclidean hexagon and identifying opposite sides. In general, if there is a boundary curve $\gamma$, if one was to glue a euclidean hemisphere by gluing the equator to $\gamma$, one would once again obtain the same conformal structure. This family of tori have different characterizations: for instance, for each boundary length they are the unique one holed tori with an isometry group of order $12$. And one could ask whether they satisfy a generalization to one holed tori of the Schmutz Schaller conjecture \ref{con:ss} mentioned earlier. In \cite{pamcmul} it is shown that there is a dense subset of these surfaces which fail to satisfy the generalization, and in \cite{pamcmuls} it is shown that if $\ell(\gamma)$ is such that $\cosh\frac{\ell(\gamma}{2})$ is a transcendental real, then it does satisfy the Schmutz Schaller conjecture.
\end{remark}

We can now explain the partial result for closed surfaces of genus 3 obtained by Schmutz Schaller and mentioned above. He shows that among all closed genus 3 surfaces with a configuration of 3 distinct systoles lying inside an embedded 1 holed tori (thus exactly like the ones explained above), the surface $M(3)$ has maximal systole. It should be said that Schmutz Schaller's detailed analysis of systole configurations are extremely useful in studying some of the synthetic geometry of Teichm\"uller and moduli spaces for low complexity surfaces. In higher complexity surfaces, there are no significantly different known ways of attacking these problems and the combinatorics of the problem become quickly out of hand.\\

To further underly the difficulty behind these problems, let us give another question which can already be found in \cite{sc98}, and which seems to wide open.

\begin{question}\label{q:growth} Is the maximal systole in genus $g+1$ greater than the maximal systole in genus $g$?
\end{question}

There is also a similar question for Riemannian surfaces which also seems to be wide open. The appropriate quantity that one studies for Riemannian surfaces is the {\it systolic ratio} and is given by
$$
\sys_g= \sup \frac{\sys^2(S)}{\area(S)}
$$
where $\sys(S)$ is the length of the shortest non-trivial closed curve on the surface $S$, and the supremum is taken among all {\it Riemannian} surfaces of genus $g$ (the notation is highly non-standard for Riemannian geometers but this is only to be able to relate the two subjects). The appropriate question here is whether $\sys_{g+1} \leq \sys_g$. Other questions and results about systolic geometry and topology can be found in \cite{kabook}.

\noindent{\it Growth of the systolic constants}\\

Systolic constants also provide an interesting growth problem. For reasons outlined in \cite{ad98,sc941}, the constants can only actually grow if one increases genus (as opposed to adding cusps). The interesting question is on how the constants $\sys(g)$ behave. For reasons outlined above, the strategy of finding the optimal constants in each genus is at the very least hopeful, so one is interested in less precise results. The rough growth was solved by Buser and Sarnak \cite{busa94} who obtained the following theorem.

\begin{theorem}[Buser-Sarnak] There exist constants $A$ and $B$ such that 
$$
A \log g < \sys(g) < B\log g.
$$
\end{theorem}

Note that by an area argument the upper bound is trivial. To show the lower bound they constructed a family of surfaces of genus $g_k$ with $g_k \to \infty$ with $k\to \infty$ and systole length with growth $\frac{4}{3} \log g_k$. Their construction relies on the use of congruence subgroups in a quaternion algebra. From their construction, you can extrapolate a full family of surfaces (i.e. a surface in every genus) with logarithmic systole growth. Since then, there have been other constructions of surfaces with large systole growth.\\

For instance in \cite{kascvi07}, Katz, Schaps and Vishne consider constructions based on congruences in other matrix groups. In particular, they show that the family of surfaces which reaches Hurwitz's bound of maximal number of automorphisms in a given genus also provides such a family (meaning a family of surfaces of genus $g_k$ with $g_k \to \infty$ with $k\to \infty$ and systole length with growth $\frac{4}{3} \log g_k$).\\

Yet another construction is essentially due to Brooks: in \cite{br99} he considers the surfaces obtained by uniformizing the surfaces coming from principal congruence subgroups of $\PSL_2 (\Z)$ mentioned above, all of them maximal in their respective signatures. Specifically, the quotients by the principle congruence subgroups give conformal structures on an underlying closed surface, and one considers the unique closed hyperbolic metric in the same conformal class. This gives a sequence of surfaces which Brooks calls the Platonic surfaces. Brooks explains how to compare surfaces with cusps with their compactifications, provided the cusps are sufficiently far apart. It seems that Brooks' theorems have not been used for this purpose, as originally the goal was to find families of surfaces with ``large" first eigenvalue of the laplacian (meaning uniformly bounded away from $0$), but Brooks' theorems imply that one only has to compute the systole in the non-compact surfaces, which is straightforward. And once again, this gives a sequence of surfaces of $g_k$ with $g_k \to \infty$ with $k\to \infty$ and systole length with growth $\frac{4}{3} \log g_k$.\\

Interestingly there seems to be a gap between the multiplicative constant in the upper and lower bounds as the best known multiplicative constant in the upper bound is 2 (and this is trivial as mentioned above). There are a number of conjectures about the exact bounds. The strongest conjecture related to this is the affirmative answer to the following question.
  
\begin{question}\label{q:asy} Does maximal systole length have asymptotic growth $\frac{4}{3} \log g$?
\end{question}

Misha Katz has called this the {\it Rodin problem}. Also note that a positive answer to this question is an extremely strong statement and implies a number of partial results, as for instance the fact that these systole functions have asymptotic growth. It also implies the existence of an upper bound with $\frac{4}{3}\log g$ behavior conjectured by Schmutz Schaller \cite{sc98}.\\

Finally, note that Schmutz Schaller was also very interested in how many distinct systoles a surface could have. He showed a number of results including the exhibition of different families of surfaces, both closed and with punctures, with number of systoles growing more than linearly in the Euler characteristic $\chi$. Specifically, the best result \cite{sc97} is that an upper bound on the number of systoles cannot grow asymptotically less than $(-\chi)^{\frac{4}{3}}$. In \cite{sc942}, he claims to show that for closed surfaces, there is an upper bound of order $g^2$, but the proof is not very convincing. Finally, he conjectures \cite{sc97} that there should be an upper bound on with growth $(-\chi)^{\frac{4}{3}}$. As systoles do not pairwise intersect more than once, one could be interested in the maximal number of simple closed geodesics with this property, but not much seems to be known.

\subsection{Bers' constants}

Another length function which has proved useful in the study of Teichm\"uller space is the length of the shortest pants decomposition of a surface. For a given surface $S\in \T_{g,n}$ and a pants decomposition $\PP$ of $S$, we define the length of $\PP$ as 
$$
\ell(\PP)=\max_{\gamma \in \PP} \ell(\gamma).
$$
We denote $\BB(S)$ the length of a shortest pants decomposition of the given surface $S$. The quantity $\BB_{g,n}$ is defined as
$$
\BB_{g,n}=\sup_{S\in \T_{g,n}} \ell(\BB(S)).
$$
This quantity is a finite quantity by a theorem of Bers \cite{be74,be85} and the constants $\BB_{g,n}$ are generally called Bers' constants. Depending on what one wants to use these results for, just the existence of the constants is good enough. For instance, it plays a crucial role in the proof of Brock's theorem \cite{br03} that Teichm\"uller space endowed with the Weil-Peterson metric is quasi-isometric to the pants complex. Building on work of Wolpert (see for instance \cite{wohand}), Brock covers Teichm\"uller space with regions corresponding to when given marked pants decompositions are short (called Bers regions in reference to Bers' constants) and then sends these regions to the topological pants decompositions. However, if one wants to say something explicit (like for instance give bounds on the quasi-isometric constants of Brock's theorem) then one needs bounds on Bers' constants. Other uses include bounds on the number of non-isometric isospectral surfaces \cite{bubook}.\\

An explicit bound can be extracted from a proof of Bers' theorem in \cite{abbook}. Buser's investigations led to a number of bounds \cite{buhab,buse92}, where best lower and upper bound for closed surfaces of genus $g$ can be found in \cite[Theorems 5.1.3, 5.1.4]{bubook}.

\begin{theorem}[Buser]\label{thm:buser}
Bers' constants satisfy $\sqrt{6g}-2 \leq \BB_{g,0} \leq 6 \sqrt{3\pi} (g-1)$.
\end{theorem}

Note that the first linear upper bounds were obtained in collaboration with Sepp{\"a}l{\"a} \cite{buse92} where they also show that on a surface with a reflexion, one can choose such a pants decomposition so that it remains globally invariant by the reflexion. As in the case of the systole, one can ask about exact values and whether the ``$\sup$" in the definition above can be replaced by a ``$\max$". This is not immediate but does follow nonetheless from Mumford's compactness theorem as in the case of systoles \cite{pababers}. As it is a short example of how some of the basic tools described previously are used, a proof is provided.

\begin{property}\label{property:maxbers1}
There exists a surface $S_{\max}\in \T_{g,n}$ such that $\BB(S_{\max}) = \BB_{g,n}$. Furthermore $\sys(S_{\max})\geq s_{g,n}$ where $s_{g,n}>0$ is a constant that only depends on $g$ and $n$.
\end{property}

\begin{proof}
Given a surface $S \in \T_{g,n}$, by the collar lemma any simple closed geodesic that crosses a geodesic of length $\ell$ has length at least $2\,\arcsinh{1\over \sinh\frac{\ell}{2}}.$ A surface $S \in \T_{g,n}$ with a short enough systole $\sigma$ (shorter than a computable constant $s_{g,n}$) has the property that any simple closed geodesic that crosses its systole has length at least $2 \BB_{g,n}$. By the collar lemma again,  because $\ell(\sigma)< 2 \arcsinh 1$, all systoles of $S$, if there are several, are disjoint. Thus a shortest pants decomposition of $S$ necessarily contains all the systoles of $S$. By using the length expansion lemma explained above, one can increase the length of all the systoles at least up until $s_{g,n}$ to obtain a new surface $S'$ such that the lengths of all simple closed geodesics disjoint from the systoles increase (strictly). In particular $\BB(S') > \BB(S)$ and $\sys(S')=s_{g,n}$.  Thus we have moved to the thick part of moduli space while increasing the Bers' constant. The thick part of moduli space being compact \cite{mu71}, it suffices to find the $\sup$ for $\BB_{g,n}$ on a compact set. As $\BB$ is a continuous function over moduli space and this proves the existence of a $S_{\max}$.
\end{proof}

Along similar lines, one can show that that Bers' constants satisfy certain inequalities which one be happy to know in the case of systoles \cite{pababers}.

\begin{property}\label{property:maxbers3}
The following inequalities hold:
\begin{enumerate}[a)]
 \item $ \BB_{g,n+1}>  \BB_{g,n}$, 
\item $ \BB_{g,n}>  \BB_{g-1,n+2}$,
\item $ \BB_{g+1,n}>  \BB_{g,n}$.
\end{enumerate}
\end{property}
Thus one can ask which surfaces attain extrema, and if the corresponding surfaces have interesting geometry. Gendulphe \cite{gen08} has recently identified the exact value of $\BB_{2,0}$, and as in the case of systoles, this is the only known value for closed surfaces. 

\begin{theorem}[Gendulphe]\label{thm:bers2}
The constant $\BB_{2,0}$ is determined by $\cosh\frac{\BB_{2,0}}{12} = x_0$ where $x_0$ is the unique solution greater than $1$ of the equation 
$$
32x^5-32x^4-24x^3 + 24x^2 - 1 = 0.
$$ 
The surface that realizes this constant is unique up to isometry.
\end{theorem}

Note that the surface that realizes this bound is not the Bolza curve. However, in contrast with the systolic constant case, even the rough asymptotic growth of $\BB_{g,n}$ is not known. Many of the known results are due to Buser (see \cite{bubook}). Buser has also conjectured what the rough growth should be.

\begin{conjecture}[Buser] There exists a universal constant $C$ such that $\BB_{g,n}\leq C\sqrt{g+n}$. 
\end{conjecture}

With Florent Balacheff, we've obtained a positive answer to this question for punctured spheres \cite{pababers} building on work of Balacheff and Sabourau \cite{basadias}, but the general case remains wide open.

\bibliographystyle{plain}

\def\cprime{$'$}

\end{document}